\documentclass[12pt,amstex]{amsart}

\usepackage{mathptmx}
\usepackage{mathrsfs}

\usepackage{verbatim}
\usepackage{url}
\usepackage[all]{xy}
\usepackage{color}

\usepackage[colorlinks=true,citecolor=blue]{hyperref}

\usepackage{stmaryrd}
\usepackage{epsfig}
\usepackage{amsmath}
\usepackage{amssymb}

\usepackage{amscd}
\usepackage{graphicx}
\usepackage{pstricks}
\usepackage{tocvsec2}

\theoremstyle{plain}
\newtheorem{thm}{Theorem}[section]
\newtheorem{lem}[thm]{Lemma}

\newtheorem{prop}[thm]{Proposition}
\newtheorem{ques}{Question}

\theoremstyle{definition}
\newtheorem{de}[thm]{Definition}
\newtheorem{exam}[thm]{Example}
\newtheorem{rem}[thm]{Remark}

\numberwithin{equation}{section}

\def \N {\mathbb{N}}

\def \Z {\mathbb{Z}}
\def \R {\mathbb{R}}
\def \Q {\mathbb{Q}}
\def \bbh {\mathbb{H}}

\def \O {\mathcal{O}}

\def \F {\mathcal F}
\def \G {\mathcal{G}}
\def \B {\mathcal B}

\def \P {\mathcal P}

\def \X {\mathcal{X}}

\def \a {\alpha }
\def \b {\beta}
\def \ep {\epsilon}
\def \d {\delta}
\def \D {\Delta}

\def \lra {\longrightarrow}

\begin{document}
\title{Multiply minimal points for the product of iterates\\
\medskip
({\small \rm Dedicated To
Professor Benjamin Weiss})
}

\author{Wen Huang}
\author{Song Shao}
\author{Xiangdong Ye}

\address{CAS Wu Wen-Tsun Key Laboratory of Mathematics, and
Department of Mathematics, University of Science and Technology of China, Hefei, Anhui, 230026, P.R. China}

\email{wenh@mail.ustc.edu.cn}
\email{songshao@ustc.edu.cn}
\email{yexd@ustc.edu.cn}


\thanks{This research is supported by NNSF of China (11971455, 12031019, 11731003,  11571335 ).}


\begin{abstract}
The multiple Birkhoff recurrence theorem states that for any $d\in\mathbb N$, every system $(X,T)$ has a
multiply recurrent point $x$, i.e. $(x,x,\ldots, x)$ is recurrent under $\tau_d=:T\times T^2\times \ldots \times T^d$.
It is natural to ask if there always is a multiply minimal point, i.e. a point $x$ such that $(x,x,\ldots,x)$
is $\tau_d$-minimal. A negative answer is presented
in this paper via studying the horocycle flows.

However, it is shown that for any minimal system $(X,T)$ and any non-empty open set $U$,
there is $x\in U$ such that $\{n\in{\mathbb Z}: T^nx\in U, \ldots, T^{dn}x\in U\}$ is piecewise syndetic; and
that for a PI minimal system, any $M$-subsystem of $(X^d, \tau_d)$ is minimal.


\end{abstract}

\maketitle

\markboth{Multiply minimal points for the product of iterates}{S. Shao W. Huang and X.D. Ye}


\tableofcontents 


\section{Introduction}

By a {\em topological dynamical system} (t.d.s. or system for short) we mean a pair $(X,T)$,
 where $X$ is a compact metric space (with a metric $\rho$) and $T:X \to X$ is a homeomorphism. 
For $x\in X$, the {\it orbit of $x \in X$} is defined by $\O(x,T)=\{T^nx: n\in \Z\}$. 
A t.d.s. $(X,T)$ is {\it transitive} if for any non-empty open sets $U$ and $V$, there is $n\in\Z$
with $U\cap T^{-n}V\not=\emptyset$. Any point in a transitive system with dense orbit is called a {\it transitive point}.
A point $x$ is {\it recurrent} if there is a sequence $\{n_i\}$ of $\Z$ with $n_i\rightarrow \infty$ and $T^{n_i}x\lra x$.
A t.d.s. $(X,T)$ is {\it weakly mixing} if the product system $(X\times X,T\times T)$ is transitive.
A t.d.s. $(X,T)$ is {\it minimal} if $\O(x,T)$ is dense in $X$ for every $x\in X$. A point $x \in X $
is	{\it minimal} or {\it uniformly recurrent} if the subsystem $(\overline{\O (x,T)},T)$ is minimal.

\medskip
The multiple Birkhoff recurrence theorem can be deduced from the multiple recurrence theorem
of Furstenberg \cite{F77} which was proved by using deep ergodic tools. It is Furstenberg and Weiss \cite{FW} who
presented a topologically dynamical proof  of the theorem. 
That is, using purely topological tools they showed that if $(X,T)$ is a t.d.s. 
and $d\in \N$, then there exist some $x\in X$ and some sequence $\{n_k\}$ of $\N$ with $n_i\to +\infty$ and $T^{in_k}x\to x$
simultaneously for $i=1,\ldots,d$. We refer to this theorem as {\it the multiple Birkhoff recurrence theorem}. And such a
point $x$ is called a {\em multiply recurrent point}, i.e., $x^{(d)}=:(x,x,\ldots, x)$ is a recurrent point
of $\tau_d=:T\times T^2\times \ldots \times T^d$.

Recall that a t.d.s. $(X,T)$ is {\it distal} if for any $x\not=y\in X$, $\liminf_{n\in \Z}\rho(T^nx, T^ny)>0$.
For a distal system $(X,T)$  and any $x\in X$, $(x,x,\ldots, x)$ is not only a recurrent point of $\tau_d$,
but also a uniformly recurrent point of $\tau_d$. Thus in \cite[Page 231]{F81},
Furstenberg asked a question about multiple uniform recurrence:

\medskip

 {\em ``We will see in the next section that the latter property (means multiple recurrence) always holds for
some point of any system $(X,T)$. On the other hand we do not know if there always exists a point
$x$ such that $(x,x,\ldots,x)$ is a uniformly recurrent point for $T\times T^2\times \ldots\times T^r$." }

\medskip

We shall call such a point $x$ a {\em multiply minimal point}, i.e. $(x,x,\ldots, x)$ is a minimal point of $\tau_d$.
As mentioned above, for a distal system all points are multiply minimal.
Recall that a system $(X, T)$ is {\it doubly minimal} if the orbit closure of any pair
$(x,y)$ under $T\times T$ is either all of $X\times X$ or the graph of a power of $T$.
For a doubly minimal system $(X,T)$, all points are multiply minimal, in fact,
$(X^d, \tau_d)$ is minimal for all $d\in \N$ \cite{AM85}.
Weiss \cite{Weiss95} showed that any zero entropy ergodic measure theoretic system has a uniquely ergodic model which is
doubly minimal.

Although there are so many positive evidences showing the existence of multiply minimal
points, in this paper we demonstrate that there exists a system $(X,T)$ without any multiply
minimal points, which gives a negative answer to the question above. Namely, as one of the main results
we have

\medskip
\noindent {\bf Theorem A:} {\it There is a minimal weakly mixing system $(X,T)$ without any multiply minimal
points. In fact, for all $x\in X$, $(x,x)$ is a transitive point of
$(X\times X, T\times T^2)$ but not minimal.}

\medskip

Recall that a subset $S$ of $\Z$ is {\it syndetic} if it has a bounded gap,
i.e. there is $N\in \N$ such that $\{i,i+1,\cdots,i+N\} \cap S \neq
\emptyset$ for every $i \in {\Z}$. $S$ is {\it thick} if it
contains arbitrarily long runs of 
integers.
A subset $S$ of $\Z$ is {\it piecewise syndetic} if it is an
intersection of a syndetic set with a thick set. A classic result states that $x$ is a minimal
point if and only if $N_T(x,U):= \{n\in \Z: T^nx\in U\}$ is syndetic for any neighborhood
$U$ of $x$ \cite{GH}. Since we have showed that there is a system $(X,T)$ without
any multiply minimal points, it is natural to ask a weaker question: let $(X,T)$ be a t.d.s. and $d\in \N$.
Is there a point $x\in X$ such that $(x,x,\ldots,x)$ is piecewise syndetic recurrent? That is, for any neighborhood
$U$ of $x$, is the set $N_{\tau_d}(x^{(d)}, U^d)$ piecewise syndetic?

\medskip

Note that we always can find some point $x$ in a t.d.s. $(X,T)$, which is multiply positive upper density recurrent.
Furstenberg \cite{F77} proved
Szemer\'edi's theorem via the following multiple recurrence theorem: let $T$ be
a measure preserving transformation on a probability space
$(X,\X,\mu)$, then for every integer $d \ge 1$ and $A\in \mathcal{X}$ with positive
measure,
\begin{equation*}
    \liminf_{N\to \infty} \frac{1}{N}\sum_{n=0}^{N-1}
    \mu(A\cap T^{-n}A\cap T^{-2n}A\cap \ldots \cap T^{-dn}A)>0.
\end{equation*}
Using this result, one can prove that for a t.d.s. $(X,T)$ and any $T$-invariant Borel probability measure $\mu$
on $X$, there exists a Borel subset $X_0$ of $X$ with $\mu(X_0)=1$ such that for every $x\in X_0$, every $d\in \N$ and
every neighborhood $U$ of $x$ the set $N_{\tau_d}(x^{(d)}, U^d)$
has positive upper density \footnote{The upper density of $A$ is defined by $\overline d(A)=\limsup_{N \to \infty}\frac{
|A \cap \{-N,\cdots,-1,0,1\cdots,N-1\}|}{2N+1}.$}, for details see \cite{KLOY}.

\medskip

We have the following result related to the weak question and we conjecture that it is sharp (see the remark
following Theorem C below).

\medskip
\noindent {\bf Theorem B:} {\it Let $(X,T)$ be a minimal system and $d\in \N$. Then for any non-empty open
subset $U$ of $X$, there is some $x\in U$ such that
$N_{\tau_d}(x^{(d)}, U^d)$ is piecewise syndetic.}

\medskip
Note that there are a minimal system $(X,T)$ and some $\d>0$ such that for any open subset $U$ of $X$
with ${\rm diam}(U)<\d$, for any $x\in U$, the set $N_{\tau_d}(x^{(d)}, U^d)$ is not syndetic (Theorem \ref{not-syndetic}).

\bigskip
The class of minimal PI systems (see Section 4 for the definition) is an important subclass of the class of minimal systems. Recall that
a t.d.s. is an {\it $M$-system} if it is transitive and has a dense set of minimal points. van der Woude \cite{Wo85}
showed that a minimal t.d.s. $(X,T)$ is PI if and only if any $M$-subsystem in $(X\times X, T\times T)$ is minimal,
if and only if for any $d\ge 2$ any $M$-subsystem of $(X^d, T^{(d)})$ is minimal, where
$T^{(d)}=T\times \ldots \times T$ ($d$-times). We have

\medskip
\noindent {\bf Theorem C:} {\it Let $(X,T)$ be a PI minimal system and $d\ge 2$. Then
any $M$-subsystem of $(X^d, \tau_d)$ is minimal.}

\medskip
The reason why we conjecture Theorem B is sharp can be explained as follows: we believe that there is also a minimal
PI system without multiply minimal points. For such a system, by Theorem C, there is no $x$ such that $(x,x)$ is piecewise
syndetic recurrent, since piecewise
syndetic recurrence is equivalent to the existence of an $M$-system.

\medskip
We organize the paper as follows. In Section 2, we prove Theorem A. In Section 3, we prove Theorem B.
The proof of Theorem C is carried out in Section 4.

\medskip
\noindent {\bf Acknowledgements:} We would like to thank J. Li for bringing to our attention the paper of Furstenberg \cite{F81},
J. Qiu for a careful reading and E. Glasner for many useful comments.

\section{Multiple recurrence in horocycle flows}

In this section, we prove Theorem A. In fact, we will prove the following result:

\begin{thm}
There is some horocycle system $(X,T)$ such that for all $p \neq q\in \N$ with $(p,q)=1$ and for any $x\in X$, $(x,x)$ is a transitive but not minimal point of $(X\times X,T^p\times T^q)$.
\end{thm}

To this aim, we need to recall some basic facts about horocycle systems.

\subsection{Horocycle flows}\
\medskip

Let $G=SL(2, \R)$, and let $m$ be the Haar measure on $G$. Let $\Gamma_0$ be a lattice of $G$. For $t,s\in \R$ let
$$h_t=
\begin{pmatrix}
1 & t \\
0 & 1
\end{pmatrix},\quad
g_s=
\begin{pmatrix}
e^{-s} & 0 \\
0 & e^s
\end{pmatrix}.
$$
Then $(X,\B(X),\mu_0, \{h_t\}_{t\in \R})$ is called the {\em horocycle flow} on $X=G/\Gamma_0$,
 where $\mu_0=\mu_{\Gamma_0}$ is the unique $G$-invariant probability measure on $X$.
 And $(X,\B(X),\mu_0,\{g_s\}_{s\in \R})$ is called the {\em geodesic flow}. It is easy to verify the following important relation between them:
\begin{equation}\label{s1}
  g_sh_tg^{-1}_s=h_{e^{-2s}t}, \ \forall t,s\in \R.
\end{equation}
If we denote $h=h_1$, then it is well known that $(X,\B(X),\mu_0,h)$ is a strictly ergodic system \cite{F73}.

\medskip

Let $\G_c$ be the family of all co-compact discrete subgroups of $G$. For $\Gamma\in \G_c$ denote
$$X_\Gamma=G/\Gamma, \ \mu_\Gamma=\text{the $G$-invariant probability measure on $X_\Gamma$ }.$$

Recall that $\Gamma$ and $\Gamma_0$ are {\em commensurable} if $\Gamma\cap \Gamma_0$ has finite index in both $\Gamma$ and $\Gamma_0$.
Let
$${\rm COM}(\Gamma)=\{\a\in G: \Gamma^\a\triangleq\a\Gamma \a^{-1} \ \text{is commensurable with $\Gamma$}\}.$$
By Margulis theorem \cite[Page 125]{Zimmer},
$\Gamma$ is arithmetic if and only if ${\rm COM}(\Gamma)$ is dense in $G$.

\subsection{Arithmetic subgroups of $SL(2,\R)$ via quaternion algebras}\
\medskip

In this subsection we introduce arithmetic subgroups of $SL(2,\R)$ via quaternion algebras. Refer to \cite[Section 6.2]{Morris} for more details.

\begin{de}
For any field $F$, and any nonzero $a, b \in  F$, the corresponding
\textit{quaternion algebra} over $F$ is the ring
\begin{equation}\label{}
  \bbh_F^{a,b}=\{x_0+x_1{i}+x_2{j} +x_3{ k}: x_0,x_1,x_2,x_3\in F\},
\end{equation}
where addition is defined in the obvious way, and multiplication is determined by the relations
$$i^2=a,\  j^2=b, \ ij=k=-ji.$$
Note that
$$k^2=(-ji)(ij)=-aj^2=-ab.$$

The \textit{reduced norm} of $\a =x_0+x_1i+x_2j+x_3 k\in \bbh^{a,b}_F$ is
\begin{equation}\label{}
 {\rm  N_{red}}(\a)=\a\overline{\a}=x_0^2-ax_1^2-bx_2^2+abx_3^2 \in F,
\end{equation}
where
$$\overline{\a}=x_0-x_1i-x_2j-x_3k$$
is the \textit{conjugate} of $\a$. Note that $\overline{\a \b}=\overline{\b}\overline{\a}$.
\end{de}

\begin{rem}
\begin{enumerate}
  \item $\bbh^{-1,-1}_\R=\bbh$ is the classical quaternion algebra.
  \item $\bbh^{t^2a,t^2b}_F\cong \bbh^{a,b}_F$ for any nonzero $a,b,t\in F$. An isomorphism is given by $1\mapsto 1, i\mapsto ti, j\mapsto tj, k\mapsto t^2k$.

      \item ${\rm N_{red}}(\a\b)={\rm N_{red}}(\a)\cdot {\rm N_{red}}(\b)$, for $\a, \b \in \bbh^{a,b}_F$.
\end{enumerate}
\end{rem}

\begin{rem}\label{rem1}
We have for any nonzero $a,b\in F$,
$$\bbh^{a^2,b}_F\cong {\rm Mat}_{2\times 2}(F).$$

Define an isomorphism $\phi: \bbh^{a^2,b}_F\rightarrow  {\rm Mat}_{2\times 2}(F)$ by
\begin{equation}\label{}
  \phi(1)=\begin{bmatrix}
1 & 0 \\
0 & 1
\end{bmatrix}, \ \phi(i)=
  \begin{bmatrix}
a & 0 \\
0 & -a
\end{bmatrix}, \ \phi(j)=
\begin{bmatrix}
0 & 1 \\
b & 0
\end{bmatrix},\  \phi(k)=
\begin{bmatrix}
0 & a \\
-ba & 0
\end{bmatrix}.
\end{equation}
Hence for $\a =x_0+x_1i+x_2j+x_3 k$
$$\phi(\a)=\phi(x_0+x_1i+x_2j+x_3k)=\begin{bmatrix}
x_0+x_1a & x_2+x_3a \\
b(x_2-x_3a) & x_0-x_1 a
\end{bmatrix}.$$
And
$$\det(\phi(\a))=x_0^2-a^2x_1^2-bx_2^2+a^2bx_3^2={\rm N_{red}}(\a)$$
\end{rem}

\begin{thm}\cite[Proposition 6.2.4.]{Morris}\label{thm-Morris}
Fix positive integers $a$ and $b$, and let
$$SL(1,\bbh^{a,b}_\R)=\{g\in \bbh^{a,b}_\R: {\rm N_{red}}(g)=1\}.$$
Then
\begin{enumerate}
  \item $SL(1,\bbh^{a,b}_\R) \cong SL(2,\R)$.
  \item $SL(1,\bbh^{a,b}_\Z)=\{g\in \bbh^{a,b}_\Z: {\rm N_{red}}(g)=1\}$ is an arithmetic subgroup of $SL(1,\bbh^{a,b}_\R)$.
  \item the following are equivalent:
\begin{enumerate}
  \item $SL(1,\bbh^{a,b}_\Z)$ is co-compact in $SL(1,\bbh^{a,b}_\R)$.
  \item $(0,0,0,0)$ is the only integer solution $(x_0,x_1,x_2,x_3)$ of the Diophantine equation
\begin{equation}\label{h6}
  w^2-ax^2-by^2+abz^2=0
\end{equation}
  \item Every nonzero element of $\bbh^{a,b}_\Q$ has a multiplicative inverse.
\end{enumerate}
\end{enumerate}
\end{thm}

\begin{rem}
By Remark \ref{rem1}, for positive integers $a$ and $b$ the isomorphism in Theorem \ref{thm-Morris}-(1) is given by $\phi: \bbh^{a,b}_\R\rightarrow  {\rm Mat}_{2\times 2}(\R)$ with
\begin{equation}\label{s2}
  \phi(1)=\begin{bmatrix}
1 & 0 \\
0 & 1
\end{bmatrix}, \ \phi(i)=
  \begin{bmatrix}
\sqrt{a }& 0 \\
0 & -\sqrt{a}
\end{bmatrix},\ \phi(j)=
\begin{bmatrix}
0 & 1 \\
b & 0
\end{bmatrix}, \ \phi(k)=
\begin{bmatrix}
0 & \sqrt{a} \\
-b\sqrt{a} & 0
\end{bmatrix}.
\end{equation}
For  $\a=x_0+x_1i+x_2j+x_3k\in \bbh^{a,b}_\R$, we have
\begin{equation}\label{s3}
  \phi(\a)=\begin{bmatrix}
x_0+x_1\sqrt{a} & x_2+x_3\sqrt{a} \\
b(x_2-x_3\sqrt{a}) & x_0-x_1\sqrt{a}
\end{bmatrix}.
\end{equation}
\begin{equation}
  \det(\phi(\a))= x_0^2-ax_1^2-bx_2^2+abx_3^2={\rm N_{red}}(\a).
\end{equation}
\begin{equation}\label{s4}
  {\rm Trace}(\a) \triangleq {\rm Trace}(\phi(\a))=2x_0.
\end{equation}

Thus
$$G=SL(2,\R)=\phi(SL(1,\bbh^{a,b}_\R)).$$

\end{rem}

\begin{rem}
Every cocompact, arithmetic subgroup of $SL(2,\R)$
appears in Theorem \ref{thm-Morris} up to commensurability and
conjugates \cite[Proposition 6.2.6 ]{Morris}.
\end{rem}

\subsection{A Lemma on the commensurator subgroup}\
\medskip

Let $A(\Q)=\bbh^{a,b}_\Q$ and $A(\Z)=\bbh^{a,b}_\Z$, where $a,b\in \N$.
Let $A_1(\Z)$ be the integral unit group:
\begin{equation*}\label{}
 A_1(\Z)=SL(1,\bbh^{a,b}_\Z)=\{\a\in A(\Z)=\bbh^{a,b}_\Z: {\rm N_{red}}(\a)=1\}.
\end{equation*}
We assume that the condition of Theorem \ref{thm-Morris}-(3) holds. Then
$\Gamma_0=\phi(A_1(\Z))$ is co-compact lattice in $SL(1,\bbh^{a,b}_\R)$. Let $\Gamma$
be commensurable with $\Gamma_0$. Then the commensurator of $\Gamma$ consists of the $\Q$ points. That is,

\begin{thm}\cite{BSZ, PR}\label{thm-COM}
\begin{equation}\label{}
{\rm COM}(\Gamma)=\left\{ \frac{\phi(\a)}{(\det \a)^{\frac 12}}: \a\in A^+(\Q)\right\},
\end{equation}
where $A^+(\Q)=\{\a\in A(\Q): {\rm N_{red}} (\a)>0.\}$
\end{thm}


\subsection{Minimal points of product systems}\
\medskip

Let $a=2, b=3$, i.e. we consider $\bbh^{2,3}_\R$. Now
$$i=\sqrt{2},\  j=\sqrt{3},\ k=ij=\sqrt{6}.$$

The following lemma can be gotten easily from the Hasse-Minkowski theory \cite[Chapter 1]{BS}. For completeness, we include a proof.

\begin{lem}\label{lem-Dioph}
$(0,0,0,0)$ is the only integer solution $(x_0,x_1,x_2,x_3)$ of the Diophantine equation
\begin{equation}\label{s5}
  w^2-2x^2-3y^2+6z^2=0.
\end{equation}
\end{lem}

\begin{proof}
If $(0,0,0,0)$ is not the only integer solution, then let $(w_0, x_0,y_0,z_0)\in \Z^4\setminus \{(0,0,0,0)\}$ be a solution such that
\begin{equation*}
  \begin{split}
     &\quad |w_0|+|x_0|+|y_0|+|w_0|\\
     &=\min \{|w|+|x|+|y|+|w|:  w^2-2x^2-3y^2+6z^2=0, (w,x,y,z)\in \Z^4\setminus \{(0,0,0,0)\}\}.
   \end{split}
\end{equation*}
Since $w_0^2-2x_0^2-3y_0^2+6z_0^2=0$, we have
$w^2_0-2x_0^2 \equiv 0 \pmod{3}$. But $w_0^2 \equiv 0 \ \text{or }\ 1 \pmod{3}$ and  $2x_0^2 \equiv 0 \ \text{or }\ 2 \pmod{3}$. It follows that  $w_0 \equiv 0  \pmod{3}$ and $x_0 \equiv 0  \pmod{3}$. Let $w_0=3w_1$ and $x_0=3x_1$. Then we have
$$y_0^2-2z_0^2-3w_1^2+6x_1^2=0.$$
Similarly, we have  $y_0 \equiv 0  \pmod{3}$ and $z_0 \equiv 0  \pmod{3}$. Let $y_0=3y_1$ and $z_0=3z_1$.
Then $$w_1^2-2x_1^2-3y_1^2+6z_1^2=0.$$
But $|w_1|+|x_2|+|y_1|+|w_1|<|w_0|+|x_0|+|y_0|+|w_0|$, a contradiction!
\end{proof}

By Theorem \ref{thm-Morris} and Lemma \ref{lem-Dioph}, the group
$\Gamma_0=\phi(SL(1,\bbh^{2,3}_\Z))$ is co-compact.

\begin{equation*}\label{}
\begin{split}
  \Gamma_0& =\left\{\begin{pmatrix}
x_0+x_1 i & x_2+x_3 i \\
3(x_2-x_3i) & x_0-x_1i
\end{pmatrix}: x_0,x_1,x_2,x_3\in \Z, \ x_0^2-2x_1^2-3x_2^2+6x_3^2=1 \right\}\\
&= \left\{\begin{pmatrix}
x_0+x_1\sqrt{2} & x_2+x_3\sqrt{2} \\
3(x_2-x_3\sqrt{2}) & x_0-x_1\sqrt{2}
\end{pmatrix}: x_0,x_1,x_2,x_3\in \Z, \ x_0^2-2x_1^2-3x_2^2+6x_3^2=1 \right\}.
\end{split}
\end{equation*}

Recall that for $\a\in G,$ and $\Gamma$,
$$\Gamma^{\a}=\a\Gamma \a^{-1}.$$


The following result gives a condition when a point in the product system can be minimal. Recall that $h=h_1$, and $\G_c$ is the family of all co-compact discrete subgroups of $G$.

\begin{thm}[Glasner-Weiss]\cite{GW94}\label{thm-GW}
Let $\Gamma,\Gamma'\in \G_c$. If $(g\Gamma,g'\Gamma')$ is a minimal point of $(X_\Gamma\times X_{\Gamma'}, h\times h)$ then either $(X_\Gamma\times X_{\Gamma'}, h\times h)$ is minimal or for some $s\in \R$, $\a=g^{-1}h_s g'$ is such that ${\Gamma'}^\a$ and $\Gamma$ are commensurable.

In particular, for $\Gamma=\Gamma'=\Gamma_0$,  a point  $(g\Gamma_0, g'\Gamma_0)$ is minimal if and only if $$\a=g^{-1}h_s g'\in {\rm COM}(\Gamma_0)$$ for some $s\in \R$.
\end{thm}

The notion of joining was introduced by Furstenberg in \cite{F67}.
Let $(X_i, \mu_i,T_i), i = 1,\ldots, k$, be measure preserving systems and let $(Y_i,\nu_i,S_i)$ be corresponding factors, and
$\pi_i:X_i\rightarrow Y_i$ the factor maps. A measure $\nu$ on
$Y=\prod_i Y_i$ defines a {\em joining} of the measures on $Y_i$ if
it is invariant under $S_1\times \ldots \times S_k$ and maps onto
$\nu_j$ under the natural map $\prod_i Y_i\rightarrow Y_j$. When $S_1=\ldots=S_k$, we then say that
$\nu$ is a $k$-fold selfjoining, and just selfjoining when $k=2$.

\begin{de}\cite[1.5.29]{Morris2005}
Let $\Gamma$ be a lattice in $G=SL(2,\R)$.
For some $g\in G$, let $\Gamma'=\Gamma\cap (g\Gamma g^{-1})$ and assume that $\Gamma'$ has finite index in $\Gamma$. There are two natural covering maps from $X'=G/\Gamma'$ to $X=G/\Gamma$:
$\Psi_1(x\Gamma')=x\Gamma$ and $\Psi_2(x\Gamma')=xg\Gamma$. Define
$$\Psi: X'\rightarrow X\times X, \ \Psi(x)=\big(\Psi_1(x),\Psi_2(x)\big).$$
Let $\mu_{\Gamma'}$ be the $G$-invariant probability measure on $X'$.
The measure $\hat{\mu}=\Psi_*\mu_{\Gamma'}$ is a self-joining of $X$, and is called a \textbf{finite-cover self-joining}.
\end{de}

\begin{thm}[Ratner Joinings Theorem]\cite{R83},\cite[1.5.30]{Morris2005}\label{thm-Ratner-Joining}
Any ergodic selfjoining
of a horocycle flow must be either

(1) a finite cover, or

(2) the product self-joining.
\end{thm}

\subsection{No $T\times T^2$-minimal points in the diagonal}\
\medskip

Recall that $G=\phi(SL(1,\bbh^{2,3}_\R))=SL(2, \R)$ and $\Gamma_0= \phi(SL(1,\bbh^{2,3}_\Z))$ is a fixed co-compact discrete arithmetic subgroup of $G$ by Theorem \ref{thm-Morris}. For $t,s\in \R$ let
$$h_t=
\begin{pmatrix}
1 & t \\
0 & 1
\end{pmatrix},\quad
g_s=
\begin{pmatrix}
e^{-s} & 0 \\
0 & e^s
\end{pmatrix}.
$$

Let $X=X_{\Gamma_0}=G/\Gamma_0$ and $\lambda\in \R$. Let
$$T: X\rightarrow X, \quad g\Gamma_0 \mapsto h_1g\Gamma_0.$$
and
$$T^\lambda: X\rightarrow X, \quad g\Gamma_0\mapsto h_\lambda g\Gamma_0.$$

\begin{thm}\label{thm-non-minimal}
Let $(X,T)$ be defined as above, $p \neq q\in \N$ with $(p,q)=1$. Then for any $x\in X$, $(x,x)$ is not a minimal point of $T^p\times T^q$.
\end{thm}

\begin{proof}
Let $\lambda >0$ and
$$\psi_\lambda: X\rightarrow X, \quad \psi_\lambda(g\Gamma)=g_{s(\lambda)}g\Gamma,$$
where $s(\lambda)=\frac{\log \lambda}{2}$.
Then by \eqref{s1} we have the following commuting diagram:
\begin{equation*}
  \xymatrix{
  X \ar[r]^{T^\lambda}\ar[d]_{\psi_\lambda}
                & X \ar[d]^{\psi_\lambda}\\
  X\ar[r]_{T}
                & X .}
\end{equation*}
That is, $\psi_\lambda: (X, T^\lambda)\rightarrow (X, T)$ is
an isomorphism of $(X,T^\lambda)$ and $(X, T)$. Let $x=g\Gamma$. Then
$(x,x)$ is $T^p\times T^q$-minimal if and only if $(g_{s(p)}x, g_{s(q)}x)$ is $T\times T$ minimal. By Theorem \ref{thm-GW}, we have the following claim.

\medskip

\noindent \textbf{Claim:} \textit{
A point $(x,x)$ is $T^p\times T^q$-minimal if and only if there is some $\widetilde{s}=\widetilde{s}(p,q,x)\in \R$ such that
$$O_{p,q,x}\triangleq(g_{s(p)}g)h_{\widetilde{s}}(g_{s(q)}g)^{-1}\in {\rm COM}(\Gamma_0).$$
}

\medskip

Now we show that we can not find such $x$ satisfying condition in Claim.
By Theorem \ref{thm-COM}, for any $O\in {\rm COM}(\Gamma_0)$, we have the form:
$$O=\frac{\phi(\a)}{\sqrt{\det(\a)}}, \ \a=x_0+x_1 \sqrt{2}+x_2\sqrt{3}+x_3\sqrt{6}\in A^+(\Q),$$
where we may assume that $x_0,x_1,x_2,x_3\in \Z$.
We take trace:
\begin{equation}\label{h1}
  {\rm Trace}(O)=\frac{2x_0}{\sqrt{x_0^2-2x_1^2-3x_2^2+6x_3^2}},
\end{equation}
where $x_0,x_1,x_2,x_3\in \Z$ and $x_0^2-2x_1^2-3x_2^2+6x_3^2>0$.

On the other hand, let
$$O_{p,q,x}=(g_{s(p)}g)h_{\widetilde{s}}(g_{s(q)}g)^{-1}.$$
Then
\begin{equation}\label{h2}
  {\rm Trace}(O_{p,q,x}) ={\rm Trace}(g_{s(p)-s(q)})= \sqrt{\frac{p}{q}}+\sqrt{\frac{q}{p}}=\frac{p+q}{\sqrt{pq}}
\end{equation}

By Claim and \eqref{h1}, \eqref{h2}, we have that
\begin{equation}\label{h4}
\frac{p+q}{\sqrt{pq}}=\frac{2x_0}{\sqrt{x_0^2-2x_1^2-3x_2^2+6x_3^2}}.
\end{equation}
We show that \eqref{h4} can not hold.

By \eqref{h4} we have that
\begin{equation}\label{h5}
 (p+q)^2(x_0^2-2x_1^2-3x_2^2+6x_3^2)=(2x_0)^2pq.
\end{equation}

Since $(pq,p+q)=1$, $p+q|2x_0$. It follows that $(p+q)|(p-q)x_0$. Let
$$\widetilde{x_0}=\frac{(p-q)x_0}{p+q}\in \Z.$$
By \eqref{h5},
$$\widetilde{x_0}^2- 2x_1^2- 3x_2^2+6x_3^2=0.$$
By Lemma \ref{lem-Dioph}
$$\widetilde{x_0}=x_1=x_2=x_3=0,$$
and hence
$${x_0}=x_1=x_2=x_3=0.$$
This contradicts with the fact that $x_0^2-2x_1^2-3x_2^2+6x_3^2>0$.
Thus no $x\in X$ satisfies the condition in Claim. That is, for each $x\in X$, $(x,x)$ is not $T^p\times T^q$-minimal. The proof is complete.
\end{proof}

\begin{thm}\label{thm-tran}
Let $(X,T)$ be defined as above, $p \neq q\in \N$ with $(p,q)=1$. Then for any $x\in X$, $(x,x)$ is a transitive point of $(X\times X, T^p\times T^q)$.
\end{thm}

\begin{proof}
Let $\Psi_\lambda$ be as in the proof of Theorem \ref{thm-non-minimal}. Let $x\in X$ and $x_\lambda=\Psi_\lambda(x)$ for $\lambda >0$.

By Ratner Equidistribution Theorem and Ratner Measure Classification Theorem \cite{R91, R912}, for any $x\in X$ and $p,q\in \N$, $(x_p,x_q)$ is a $T\times T$-generic point for an ergodic measure $\xi$, which is a joining of $(X,T)$. By Ratner Joinings Theorem (Theorem \ref{thm-Ratner-Joining}), $\xi$ must be either a finite cover, or the product self-joining. If $\xi$ is a finite cover, then $(x_p,x_q)$ is $T\times T$-minimal and hence $(x,x)$ is $T^p\times T^q$-minimal, which contradicts with Theorem \ref{thm-non-minimal}. Thus $\xi$ is the product self-joining. By Ratner Equidistribution Theorem $(x_p,x_q)$
is a $T\times T$-generic point of $\xi$, and hence the $T\times T$-orbit closure of $(x_p, x_q)$ is the support of $\xi$ which is $X \times X$. Thus $(x_p, x_q)$ is a $T\times T$-transitive point.  Since $\Psi_p,\Psi_q$ are isomorphisms, $(x,x)$ is a $T^p\times T^q$-transitive point.
\end{proof}

\section{Multiple recurrence times}

In this section we prove Theorem B. To do this we need some preparations.

\subsection{Furstenberg families}

Let us recall some notions related to Furstenberg families (for
details see \cite{F81}). Let $\P=\P({\Z})$ be the collection
of all subsets of $\Z$. A subset $\F$ of $\P$ is a {\em
(Furstenberg) family}, if it is hereditary upwards, i.e. $F_1
\subset F_2$ and $F_1 \in \F$ imply $F_2 \in \F$. A family $\F$ is
{\it proper} if it is a proper subset of $\P$, i.e. neither empty
nor all of $\P$. If a proper family
$\F$ is closed under finite intersections, then $\F$ is called a {\it
filter}. For a family $\F$, the {\it dual family} is
$$\F^*=\{F\in\P: {\Z} \setminus F\notin\F\}=\{F\in \P:F \cap F' \neq
\emptyset \ for \ all \ F' \in \F \}.$$ $\F^*$ is a family, proper
if $\F$ is. Clearly, $(\F^*)^*=\F$ and ${\F}_1\subset {\F}_2$ implies that ${\F}_2^* \subset {\F}_1^*.$
Denote by $\F_{inf}$ the family consisting of all infinite subsets of $\Z$.
The
collection of all syndetic (resp. thick) subsets of $\Z$ is denoted by
$\F_s$ (resp. $\F_t$). Note that $\F_s^*=\F_t$ and $\F_t^*=\F_s$. Denote the set of
all piecewise syndetic sets by $\F_{ps}$. A set $F$ is called {\em thickly syndetic} if for every $N\in \N$ the positions where interval with length $N$ runs begin form
a syndetic set. Denote the set of
all thickly syndetic sets by $\F_{ts}$, and we have $\F_{ts}$ is a filter and $\F^*_{ps}=\F_{ts}, \F^*_{ts}=\F_{ps}$.

Let $\F$ be a family and $(X,T)$ be a t.d.s. We say $x\in X$ is
$\F$-{\it recurrent} if for each neighborhood $U$ of $x$, $N(x,U)\in
\F$. So the usual recurrent point is just $\F_{inf}$-recurrent one.
It is known that a t.d.s.
$(X,T)$ is an $M$-{system}
if and only if there is a transitive
point $x$ such that $N(x,U)\in \F_{ps}$ for any neighborhood $U$ of
$x$ (see for example \cite[Lemma 2.1]{HY05} or \cite[Theorem 3.1]{Ak97}).

\subsection{Some useful lemmas}

Let $x\in X$. 
For convenience, sometimes we denote the orbit closure of $x\in X$ under $T$ by  $\overline{\O}(x,T)$ or
$\overline{\O}(x)$, instead of $\overline{\O(x,T)}$. Let $A\subseteq X$, the  orbit of $A$ is defined by
$\O(A,T)=\{T^nx: x\in A,n\in \Z\}$, and is denoted by $\overline{\O}(A,T)= \overline{\O(A,T)}$.

\medskip

Let $(X,T)$ be a t.d.s., $A\subseteq X$ and $d\in \N$. Set
$$\Delta_d(A)=\{(x,x,\ldots,x): x\in A\}\subseteq X^d,$$
$$T^{(d)}=T\times \ldots\times T \ (d \ \text{times}),\ \text{and}\ \tau_d=\tau_d(T)=T\times T^2 \times \ldots \times T^{d}.$$

Let $(X,T)$ be a t.d.s. and $d\in N$. Let
$$N_d(X)=\overline{\O}(\D_d(X), \tau_d).$$
If $(X,T)$ is transitive and $x\in X$ is a transitive point, then
$$N_d({X})=\overline{\O}(x^{(d)},
\langle\tau_d, T^{(d)}\rangle),$$
is the orbit closure of
$x^{(d)}=(x,\ldots,x)$ ($d$ times) under the action of the group
$\langle\tau_d, T^{(d)}\rangle$ which is generated by $\tau_d$ and $T^{(d)}$.

\medskip

The following result was obtained by Glasner.

\begin{prop}[Glasner \cite{G94}]\label{thm-Glasner}
Let $(X,T)$ be a minimal t.d.s. Then
\begin{enumerate}
\item $(N_d(X), \langle\tau_d, T^{(d)}\rangle)$ is minimal for each $d\in\N$.

\item If in addition $(X,T)$ is weakly mixing, then there is a dense $G_\delta$ set $X_0$ of $X$ such that
for each $x\in X_0$ and each $d\in\N$, $\overline{\O}(x^{(d)},\tau_d)=X^d$.
\end{enumerate}
\end{prop}

The following Lemma \ref{lem-Bowen} is a modification of \cite[Lemma 2.1]{F} or \cite{FW}.

\begin{lem}\label{lem-Bowen}
Let $(X,T)$ be a t.d.s. Let $A\subset X$ be a subset with the property that for every $x\in A$, and $\ep>0$, there exists $y\in A$ such that
$$\{n\in \Z: \rho(T^ny, x)<\ep \}\in \F_{ps}.$$
Then for every $\ep>0$, there exists a point $z\in A$ such that
$$\{n\in \Z: \rho(T^n z, z)<\ep \}\in \F_{ps}.$$
\end{lem}

\begin{proof}
Let $\ep>0$. We will define inductively a sequence $\{z_n\}_{n=0}^\infty$, one of which will satisfy our conclusion.

Set $\ep_1=\ep/2$. Let $z_0$ be any point of $A$ and by the assumption, let $z_1\in A$ satisfy that
\begin{equation*}
  \{n\in \Z: \rho(T^n z_1, z_0)<\ep_1 \}\in \F_{ps}.
\end{equation*}
Let $n_1\in \{n\in \Z: \rho(T^n z_1, z_0)<\ep_1 \}$. Now let $\ep_2$ with $0<\ep_2<\ep_1$ be such that whenever $\rho(z,z_1)<\ep_2$,
\begin{equation*}
 \rho(T^{n_1}z, z_0)<\ep_1.
\end{equation*}
For $\ep_2$ and $z_1$, by the assumption find $z_2\in A$ with
\begin{equation*}
  \{n\in \Z: \rho(T^n z_2, z_1)<\ep_2 \}\in \F_{ps}.
\end{equation*}
Let $n_2\in \{n\in \Z: \rho(T^n z_2, z_1)<\ep_2 \}$.
Next let $\ep_3$ with $0<\ep_3<\ep_1$ be such that whenever $\rho(z,z_2)<\ep_3$,
\begin{equation*}
 \rho(T^{n_2}z, z_1)<\ep_2.
\end{equation*}
Continue in this way defining $z_0,z_1,\ldots,z_k\in A$, $n_1,n_2,\ldots,n_k\in \Z$ and $\ep_1,\ep_2,\ldots,\ep_k\in (0,\ep/2)$ such that
\begin{equation}\label{f1}
  \{n\in \Z: \rho(T^n z_i, z_{i-1})<\ep_i \}\in \F_{ps}, 1\le i\le k
\end{equation}
\begin{equation}\label{f2}
  n_i\in \{n\in \Z: \rho(T^n z_i, z_{i-1})<\ep_i \}, 1\le i\le k.
\end{equation}
We define $\ep_{k+1}>0$
with $0<\ep_{k+1}<\ep_1$ be such that whenever
\begin{equation}\label{f3}
 \rho(z,z_k)<\ep_{k+1}\ \Rightarrow  \rho(T^{n_k}z, z_{k-1})<\ep_k.
\end{equation}
By the assumption
find $z_{k+1}\in A$ with
\begin{equation}\label{f4}
  \{n\in \Z: \rho(T^n z_{k+1}, z_k)<\ep_{k+1} \}\in \F_{ps}.
\end{equation}
Let $n_{k+1}\in \{n\in \Z: \rho(T^n z_{k+1}, z_k)<\ep_{k+1} \}$.
Thus we define inductively a sequence $\{z_n\}_{n=0}^\infty$.

\medskip

For each pair $i,j\in \N$ with $i<j$, we have that
\begin{equation}\label{f5}
 \rho(T^{n_j+n_{j-1}+\ldots+n_{i+1}}z_j,z_i)<\ep_{i+1}<\ep_1=\frac{\ep}{2}
\end{equation}
Since $X$ is compact, we can find some pair $i, j$ with $i<j$ such that
\begin{equation}\label{f100}
\rho(z_i,z_j)<\frac{\ep}{2}.
\end{equation}
Together with \eqref{f5}, we have that
\begin{equation}\label{f6}
 \rho(T^{n_j+n_{j-1}+\ldots+n_{i+1}} z_j,z_j)<\frac{\ep}{2}+\frac{\ep}{2}=\ep.
\end{equation}

Furthermore, by \eqref{f1},
\begin{equation}\label{f7}
  \{n\in \Z: \rho(T^n z_j, z_{j-1})<\ep_j \}\in \F_{ps}.
\end{equation}
For each $ n\in \{n\in \Z: \rho(T^n z_j, z_{j-1})<\ep_j \}$,
by \eqref{f3},
$$\rho(T^{n+n_{j-1}+\ldots+n_{i+1}}z_j,z_i)<\ep_{i+1}<\ep_1$$
Combining this with \eqref{f100}, we have that
\begin{equation}\label{f6}
 \rho(T^{n+n_{j-1}+\ldots+n_{i+1}} z_j,z_j)<\frac{\ep}{2}+\frac{\ep}{2}=\ep.
\end{equation}
Thus
$$m+\{n\in \Z: \rho(T^n z_j, z_{j-1})<\ep_j \}\subset \{n\in \Z: \rho(T^n z_j, z_{j})<\ep \},$$
where $m=n_{j-1}+\ldots+n_{i+1}$. Since $m+\{n\in \Z: \rho(T^n z_j, z_{j-1})<\ep_{j} \}\in \F_{ps}$,
we have that
$$ \{n\in \Z: \rho(T^n z_j, z_{j})<\ep \}\in \F_{ps}.$$
The proof is complete.
\end{proof}


\subsection{Proof of Theorem B}
Now we are able to show Theorem B.


\begin{proof}[Proof of Theorem B]
Let $(X,T)$ be a minimal t.d.s. and $d\in\N$.  Let $U$ be a non-empty open subset of $X$. We want to show there is some $x\in U$ such that
\begin{equation}\label{g1}\{n\in \Z: \tau_d^nx^{(d)}\in U\times \ldots\times U \}\in \F_{ps}.
\end{equation}
Let $A=\D_d(X)$. We have the following

\medskip

\noindent{\bf Claim:} For a given $x\in X$ and $\ep>0$ there is $y\in X$ such that
$$\{n\in \Z: \widehat{\rho}(\tau_d^ny^{(d)}, x^{(d)})<\ep \}\in \F_{ps},$$
where $\widehat{\rho}\big((x_1,\ldots,x_d),(y_1,\ldots,y_d)\big)=\max_{1\le i\le d}\rho(x_i,y_i)$.

\medskip

Assume first that the Claim holds. By Lemma \ref{lem-Bowen}, for each $\ep>0$, there is $z\in X$ with
\begin{equation}\label{newyear}
\{n\in \Z: \widehat{\rho}(\tau_d^n z^{(d)}, z^{(d)})<\ep \}\in \F_{ps}.
\end{equation}

Next we show that for any non-empty open set $U$, we can find some $x\in U$ such that \eqref{g1} holds.
Let $w\in U$ and $\ep>0$ such that $B_{2\ep}(w)\subset U$, and let $V=B_\ep(w)$. As $(X,T)$ is minimal,
there is $N\in\N$ such that $X=\cup_{i=1}^NT^{-i}V$. Let $\eta>0$ with the property that
$\rho(y_1,y_2)<\eta$ implies $\rho(T^iy_1,T^iy_2)<\ep$ for each $0\le i\le dN$. By (\ref{newyear}),
there is some $z\in X$ with
$$\{n\in \Z: \widehat{\rho}(\tau_d^n z^{(d)}, z^{(d)})<\eta \}\in \F_{ps}.$$
Assume that $z\in T^{-i}V$ with $1\le i\le N$. Then $x=T^iz\in V$ and $B_\ep(x)\subset U$. Thus we have
$$\{n\in \Z: \widehat{\rho}(\tau_d^n z^{(d)}, z^{(d)})<\eta \}\subset \{n\in \Z: \widehat{\rho}(\tau_d^n x^{(d)}, x^{(d)})<\ep \}$$
which is a subset of $\{n\in \Z: \tau_d^nx^{(d)}\in U\times \ldots\times U \}.$
Hence
$$\{n\in \Z: \tau_d^nx^{(d)}\in U\times \ldots\times U \}\in \F_{ps}.$$
That is, (\ref{g1}) holds.

\medskip
Now we prove the Claim.
Let $\ep>0$ and fix a point $x\in X$.
By Proposition \ref{thm-Glasner}, $x^{(d)}$ is a $\langle\tau_d, T^{(d)}\rangle$-minimal point. Hence
\begin{equation}\label{h7}
  H=\{(n,m)\in \Z^2: \widehat{\rho}(\tau^n(T^{(d)})^mx^{(d)},x^{(d)})<\ep\}
\end{equation}
is $\Z^2$-syndetic. Thus there is some $L\in \N$ such that
\begin{equation}\label{h8}
  H+[-L,L]^2=\Z^2,
\end{equation}
where $[-L,L]=\{-L, -L+1,\ldots,L\}$ is an interval of $\Z$.
That is, for each $(n,m)\in \Z^2$, there exits $(i,j)\in [-L,L]^2$ such that $(n+i,m+j)=(n,m)+(i,j)\in H.$

For each $j\in \{-L,-L+1,\ldots, L\}$, let $$x_j^{(d)}\triangleq(T^{(d)})^jx^{(d)}=(T^jx)^{(d)},$$
and
$$A_j\triangleq\{n\in \Z: \widehat{\rho}(\tau^{n}_dx_j^{(d)}, x^{(d)})<\ep\}=\{n\in \Z: \widehat{\rho}(\tau^{n}_d(T^{(d)})^jx^{(d)}, x^{(d)})<\ep\}.$$
By \eqref{h8}, for each $n\in \Z$, there is some $(i,j)\in [-L,L]^2$ such that $$(n,0)+(i,j)\in H.$$
That is,
$$\widehat{\rho}(\tau^{n+i}(T^{(d)})^j x^{(d)},x^{(d)})<\ep.$$
Hence $n+i\in A_j$. It follows that
$$\Z=\bigcup_{-L\le i,j\le L} (A_j-i).$$
Since $\F_{ps}$ has the Ramsey property (see \cite[Theorem 1.24]{F}), there is some $i,j\in \{-L,-L+1,\ldots, L\}$ such that
$A_j-i\in \F_{ps}$. And hence
$$\{n\in \Z: \widehat{\rho}(\tau^{n}_dx_j^{(d)}, x^{(d)})<\ep\}=A_j\in \F_{ps},$$
as $\F_{ps}$ is translation invariant.
The proof is complete.
\end{proof}

By the result in the previous section, there is some minimal system $(X,T)$ such that
for any $p\neq q\in \N$, $(p,q)=1$ there is some non-empty subset $U$ of $X$ satisfying that for any $x\in U$
$$N_{T^p\times T^q}((x,x), U\times U)=\{n\in \Z: (T^p\times T^q)^n(x,x)\in U\times U\}$$
is not syndetic. In fact we have the following example:

\begin{thm}\label{not-syndetic}
There is some minimal system $(X,T)$ such that for any $p\neq q\in \N$, $(p,q)=1$ there is some $\d>0$ with
the following property: for any open subset $U$
of $X$ with ${\rm diam}(U)<\d$, for any $x\in U$,
the set
$$N_{T^p\times T^q}((x,x), U\times U)$$
is not syndetic.
\end{thm}

\begin{proof}
Let $(X,T)$ be the horocycle flow described in Theorem \ref{thm-non-minimal}. For $p\neq q\in \N$, let $\tau=T^p\times T^q$. Then by Theorem \ref{thm-tran}, $(X\times X,\tau)$ is not minimal, and for each point $x\in X$, $(x,x)$ is a $\tau$-transitive point.

\medskip
\noindent \textbf{Claim:} \textit{
For each point $x_0\in X$, there is some $\ep_0>0$ such that
for any $x\in X$ we have that
$$E_{\ep_0;x_0,x}=\{n\in \Z: \tau^n (x,x)\in U_{\ep_0}\times U_{\ep_0}\}$$
is not syndetic, where $U_{\ep_0}=B_{\ep_0}(x_0)=\{x\in X: \rho(x_0,x)<\ep_0\}$.
}

\medskip
\noindent \textit{Proof of Claim:}\
If Claim does not hold, then there is some $x_0\in X$ such that for any $\ep>0$, there is some $x_{\ep}\in X$ with $E_{\ep; x_0,x_\ep}$ being syndetic. For any fixed such $\ep$, let
$$A_\ep\triangleq \overline{\{\tau^n(x_\ep,x_\ep): n\in E_{\ep; x_0,x_\ep}\}}\subset \overline{U_\ep}\times \overline{U_\ep},$$
where $U_\ep=B_\ep(x_0)$.
Since $E_{\ep;x_0,x_\ep}$ is syndetic, there is some $L_\ep\in \N$ such that
$$\bigcup_{i=1}^{L_\ep}\tau^{-i} \overline{U_\ep}\times \overline{U_\ep}\supset \bigcup_{i=1}^{L_\ep}\tau^{-i} A_\ep\supset \overline{\O}\big((x_\ep,x_\ep), \tau \big).$$
As $(x_\ep,x_\ep)$ is $\tau$-transitive,
$$\overline{\O}\big((x_\ep,x_\ep), \tau \big)=X\times X.$$
Hence
$$ X\times X =\bigcup_{i=1}^{L_\ep}\tau^{-i} \overline{U_\ep}\times \overline{U_\ep}.$$
Thus for each $(x_1,x_2)\in X\times X$,
$$(x_1,x_2)\in  X\times X =\bigcup_{i=1}^{L_\ep}\tau^{-i} \overline{U_\ep}\times \overline{U_\ep}.$$
There is some $i\in \{1,2,\ldots, L_\ep\}$ such that
$$(x_1,x_2)\in \tau^{-i} \overline{U_\ep}\times \overline{U_\ep}.$$
It follows that
$$\overline{\O((x_1,x_2),\tau)}\cap \overline{U_\ep}\times \overline{U_\ep}\neq \emptyset.$$
Since $\ep$ is arbitrary, it follows that
$$(x_0,x_0)\in \overline{\O((x_1,x_2),\tau)}.$$
Thus
$$X\times X= \overline{\O((x_0,x_0),\tau)}\subset \overline{\O((x_1,x_2),\tau)}.$$
So for each $(x_1,x_2)\in X\times X$, $\overline{\O((x_1,x_2),\tau)}=X\times X$.
That is, $(X\times X,\tau)$ is minimal, a contradiction!
The proof of Claim is complete. \hfill $\square$

\medskip

By the Claim, for any $z\in X$, there is some $\ep_z>0$ such that for any $x\in X$,
$$E_{\ep_z; z,x}=\{n\in \Z: \tau^n (x,x)\in B_{\ep_z}(z)\times B_{\ep_z}(z)\}$$ is not syndetic.
So, $\{B_{\ep_z}(z)\}_{z\in X}$ is an open cover of $X$. Let $\d>0$ be the Lebesgue number of $\{B_{\ep_z}(z)\}_{z\in X}$. Then for any open subset $U$ of $X$ with ${\rm diam}(U)<\d$, there is some $z\in X$ such that $U\subset B_{\ep_z}(z)$. For any $x\in X$, the set
$$N_{\tau}((x,x), U\times U)=\{n\in \Z: \tau^n(x,x)\in U\times U\}\subset E_{\ep_z;z,x}$$
is not syndetic. The proof is complete.
\end{proof}

To summarize, we have the following results.

\begin{thm}\label{summry}
Let $(X,T)$ be a minimal system $d\in \N$ and $\d>0$. Let
$$E(x;\d)=\{n\in \Z: \max_{1\le i\le d} \rho (T^{in}x,x)<\d\}=\{n\in \Z: \widehat{\rho}(\tau^{n}_dx^{(d)}, x^{(d)})<\d\}.$$
\begin{enumerate}
  \item For any $\d>0$
$$X_\d^{ps}\triangleq \{x\in X: E(x;\d) \ \text{is piecewise syndetic}\}$$
is a dense set of $X$.
  \item There is some minimal system $(X,T)$ and $\d>0$ such that for each $x\in X$, $E(x;\d)$ is not syndetic.
  \item If in addition $(X,T)$ is weakly mixing, then there exists a dense $G_\delta$ subset $\Omega$ such that
  for each $x\in \Omega$ and each $\delta>0$, $E(x;\d)$ is piecewise syndetic. 
\end{enumerate}
\end{thm}
\begin{proof} It remains to show (3). Note that by Proposition \ref{thm-Glasner}, there is a dense $G_\delta$ subset of $X_0$
such that for each $x\in X_0$, $x^{(d)}$ is a transitive point of $(X^d,\tau_d)$. As $(X^d,\tau_d)$ is an $M$-system,
for each neighborhood $U$ of $x$, $N_{\tau_d}(x^{(d)}, U^d)$ is piecewise syndetic.
\end{proof}

For some discussion related to (3) of Theorem \ref{summry}, we refer to \cite{G94} and \cite{HSY-19-1}.

\subsection{Weakly mixing systems}

Let $(X,T)$ be a t.d.s. Put
$$A=\{x\in X: (x,x)\ \text{is minimal for}\ T\times T^2\}.$$

\begin{thm}\label{wm-1}
Let $(X,T)$ be a minimal weakly mixing system. Then
we have the following possibilities
\begin{enumerate}
\item $A$ is equal to $X$;
\item $A$ is not empty and it is a subset of first category in $X$;
\item $A$ is empty.
\end{enumerate}
\end{thm}
\begin{proof} It is not difficult to check that $(X\times X, T\times T^2)$ is weakly mixing.
By Proposition \ref{thm-Glasner}, there is a dense $G_\delta$ subset $X_0$ of $X$ such that for each
$x\in X_0$, $(x,x)$ is a transitive point for $T\times T^2$.

If $(X\times X, T\times T^2)$ is minimal, it is clear that $A=X$.
Now we assume that $(X\times X, T\times T^2)$ is not minimal. This implies that
$A\cap X_0=\emptyset$, and thus $A\subset X\setminus X_0$ is a first category subset.
\end{proof}

As we said in the introduction, a double minimal system satisfies Theorem \ref{wm-1}-(1). The system in Theorem A satisfies Theorem \ref{wm-1}-(3).
Now we give a system satisfying Theorem \ref{wm-1}-(2) by a modification of the double minimal systems.

\begin{exam}
There is some weakly mixing minimal system $(X,T)$ such that $(X\times X, T\times T^2)$ is not minimal and there are points $x\in X$ such that $(x,x)$ is $T\times T^2$-minimal.
\end{exam}

\begin{proof}
Let $(Y,S)$ be a doubly minimal system. Let $(X,T)=(Y\times Y, S\times S^2)$. Then it is
weakly mixing minimal \cite{AM85}. We verify that $(X,T)$ is what we need. For all $y\in Y$,
let ${\bf x}=(y,y)\in X$. We show that $({\bf x},{\bf x})$ is $T\times T^2$ minimal and its orbit is not dense.

Since $(Y,S)$ be a doubly minimal system,  $(Y^3, S\times S^2\times S^4)$ is minimal.
Thus $({\bf x},{\bf x})=(y,y,y,y)$ is minimal under action $S\times S^2\times S^2\times S^4$,
since $(y,y,y)$ is $S\times S^2\times S^4$-minimal.

Note that
\begin{equation*}
  \begin{split}
     \overline{\O}\big(({\bf x},{\bf x}), T\times T^2\big)& =\overline{\O}\big(((y,y),(y,y)), T\times T^2\big)\\& =\{\big((z_1,z_2),(z_2,z_3)\big): z_1,z_2,z_3\in Y\}\\& \neq Y^4=X\times X.
   \end{split}
\end{equation*}
Thus $(X\times X,T\times T^2)$ is not minimal.
\end{proof}

\section{The PI case}

In this section we will prove Theorem C. We start with the definition of PI systems.

\subsection{PI systems}
We first recall some definitions.
Let $(X,T)$ be a t.d.s.
A pair $(x,y)\in X\times X$ is said to be {\em proximal} if $\liminf_{n\to \infty} \rho(T^nx,T^ny)=0$.
The set of all proximal pairs of $X$ is denoted by $P(X,T)$.
An extension $\pi :
(X,T) \rightarrow (Y,S)$ is called {\em proximal} if each pair in $R_{\pi}=\{(x_1,x_2): \pi(x_1)=\pi(x_2)\}$ is proximal. An extension  $\pi$ is {\em  almost one to one} if
$\{x\in X: |\pi^{-1}(\pi(x))|=1\}$ is a dense $G_\delta$ subset.
An extension $\pi$ is an {\em equicontinuous} or {\em isometric} one if for every $\epsilon >0$ there is
$\delta >0$ such that $(x,y) \in R_{\pi}$ and $\rho (x,y)<\delta$ imply $\rho(T^nx,T^ny)<\epsilon$,
for every $n \in \Z$.
An extension $\pi$ is a {\em weakly
mixing extension} if $(R_\pi, T\times T)$ as a subsystem of the product
system $(X\times X, T\times T)$ is transitive.
An extension $\pi$ is called a {\em relatively incontractible (RIC) extension}\ if it is open and for every $n \ge 1$
the minimal points are dense in the relation
$$
R^n_\pi = \{(x_1,\dots,x_n) \in X^n : \pi(x_i)=\pi(x_j),\ \forall \ 1\le i
\le j \le n\}.
$$

We say that a minimal system $(X,T)$ is a
{\em strictly PI system} (PI means proximal-isometric) if there is an ordinal $\eta$
and a family of systems
$\{(W_\iota,w_\iota)\}_{\iota\le\eta}$
such that

\begin{enumerate}
 \item $W_0$ is the trivial system,
  \item for every $\iota<\eta$ there exists an extension
$\phi_\iota:W_{\iota+1}\to W_\iota$ which is
either proximal or equicontinuous,
  \item for a
limit ordinal $\nu\le\eta$ the system $W_\nu$
is the inverse limit of the systems
$\{W_\iota\}_{\iota<\nu}$,
  \item $W_\eta=X$.
\end{enumerate}

We say that $(X,T)$ is a {\em PI-system} if there
exists a strictly PI system $\tilde X$ and a
proximal extension $\theta:\tilde X\to X$.





Finally we have the structure theorem for minimal systems
(see Ellis-Glasner-Shapiro \cite{EGS})

\begin{prop}[Structure theorem for minimal systems]\label{structure}
Let $(X,T)$ be a minimal system. Then we have the following diagram:

$$
\xymatrix{
  X          & X_\infty \ar[d]^{\pi_\infty} \ar[l]^{\theta} \\
                & Y_\infty             }
$$
\medskip

\noindent where $X_\infty$ is a proximal extension of $X$ and a RIC
weakly mixing extension of the strictly PI-system $Y_\infty$.
The extension $\pi_\infty$ is an isomorphism (so that
$X_\infty=Y_\infty$) if and only if  $X$ is a PI-system.
\end{prop}

\subsection{The proof of Theorem C} 

For minimal systems, one can characterize PI as follows, see van der Woude \cite{Wo85}.
Let $(X,T)$ be a minimal t.d.s. Then $X$ is PI if and only if  any $M$-subsystem of $(X\times X, T\times T)$
is minimal if and only if for any $d\ge 2$ any $M$-subsystem of $(X^d, T^{(d)})$ is minimal.

\medskip
Now we are ready to show Theorem C.

\begin{proof}[Proof of Theorem C]

Fix $d\ge 2$. 
By the definition, a PI system
is a factor of a strictly PI system. The property that any $M$-subsystem of $(X^d, \tau_d)$ is minimal, is preserved by the factor map and the inverse limit. Thus it suffices to show this property can be preserved by equicontinuous extensions and proximal extensions.

\medskip

Let $\xi: (Z, T)\rightarrow (W,T)$ be an equicontinuous extension of
systems. Then for any $z\in Z$, if $\xi(z)$ is minimal in $(W,T)$,
then $z$ is minimal in $(Z,T)$ (see \cite[Chapter II, Lemma 1.2]{G76} for example).
Now let $\pi:(X,T)\rightarrow (Y,T)$ be an equicontinuous extension.
Then $\pi^{(d)}: (X^d,\tau_d)\rightarrow (Y^d,\tau_d)$ is also equicontinuous.
Assume that any  $M$-subsystem of $(Y^d, \tau_d)$ is minimal. If $(Z,\tau_d)$ is an $M$-subsystem of $(X^d, \tau_d)$,
then by result above any point of $Z$ is $\tau_d$-minimal. Thus $Z$ is minimal itself as $Z$ is transitive.

\medskip

Now let $\pi:(X,T)\rightarrow (Y,T)$ be a proximal extension. It is easy to verify that
$\pi^{(d)}: (X^d,\tau_d)\rightarrow (Y^d,\tau_d)$ is also proximal. Assume that any
$M$-subsystem of $(Y^d, \tau_d)$ is minimal. For an $M$-subsystem $(Z,\tau_d)$ of $(X^d, \tau_d)$,
we will show that it is $\tau_d$-minimal. To do this let $(x_1,x_2,\ldots,x_d)$ be a transitive point of $(Z,\tau_d)$.
Then it is piecewise syndetic recurrent
(see for example \cite[Lemma 2.1]{HY05}). That is, for ant $\ep>0$,
\begin{equation}\label{q1}
 B= \{n\in \Z: \rho(T^nx_1,x_1)<\ep, \rho(T^{2n}x_2,x_2)<\ep, \ldots, \rho(T^{dn}x_d,x_d)<\ep\}\in \F_{ps}.
\end{equation}
Let $(y_1,y_2,\ldots,y_d)=\pi^{(d)}(x_1,x_2,\ldots,x_d)$. Then $\overline{\O}((y_1,y_2,\ldots,y_d),\tau_d)$ is
an $M$-system in $(Y^d, \tau_d)$. By the hypothesis, $(y_1,y_2,\ldots,y_d)$ is $\tau_d$-minimal.
Let $(z_1,z_2,\ldots,z_d)\in X^d$ be a minimal point of $(X^d,\tau_d)$ such that
$(y_1,y_2,\ldots,y_d)=\pi^{(d)}(z_1,z_2,\ldots,z_d)$. Note that we have $\pi(x_1)=\pi(z_1),\pi(x_2)=\pi(z_2),\ldots, \pi(x_d)=\pi(z_d),$
and hence $$(x_1,z_1)\in P(X,T),(x_2,z_2)\in P(X,T^2),\ldots, (x_d, z_d)\in P(X,T^d),$$ where $P(X,T^j)$ is the set of proximal pairs of $(X,T^j), 1\le j\le d$. It follows \cite[Theorem 3.1]{Ak97} that
\begin{equation}\label{q2}
   \{n\in \Z: \rho(T^{jn}x_2,T^{jn}z_2)<\ep\}\in \F_{ts}, \quad 1\le j\le d.
\end{equation}
Since $\F_{ts}$ is a filter and $\F^*_{ps}=\F_{ts}$, by \eqref{q1} and \eqref{q2} we have
$$B\cap \bigcap_{j=1}^d \{n\in \Z: \rho(T^{jn}x_2,T^{jn}z_2)<\ep\}\neq\emptyset.$$
In particular, we have
$$\rho(T^nz_1,x_1)<2\ep,\ \rho(T^{2n}z_2,x_2)<2\ep,\ \ldots,\ \rho(T^{dn}z_d,x_d)<2\ep.$$
Since $\ep$ is arbitrary, it follows that  $(x_1,x_2,\ldots, x_d)\in \overline{\O}((z_1,z_2,\ldots, z_d),\tau_d)$.
By $\tau_d$-minimality of $(z_1,z_2,\ldots,z_d)$, $(x_1,x_2,\ldots, x_d)$ is a minimal point of $(X^d,\tau_d)$ too. Thus $(Z,\tau_d)$ is minimal. The proof is complete.
\end{proof}

\begin{rem}
Unlike van der Woude's result, the converse of Theorem C is not true. For a doubly minimal system $(X,T)$,
$(X^d, \tau_d)$ is minimal for all $d\in \N$ \cite{AM85}. A doubly minimal system is weakly mixing, and hence it is not PI.
\end{rem}




\subsection{Some open questions}
In Section 2, we have shown the existence of some minimal weakly mixing systems without multiply
minimal points. Thus we ask
\begin{ques}
Is there a PI minimal system $(X,T)$ without any multiply minimal points?
\end{ques}

We remark that if $(X,T)$ is a minimal mean equicontinuous system (which implies that $(X,T)$
is uniquely ergodic), then there is a measurable set $X_0$ of $X$ with measure $1$
such that for each $x\in X_0$, $x^{(d)}$ is $\tau_d$-minimal. For details, see \cite{CTY}.

\medskip

In the previous section we have shown that if $(X,T)$ be a minimal system and $d\in \N$,
then for any non-empty open subset $U$ of $X$, there is some $x\in U$ such that
$$N_{\tau_d}(x^{(d)}, U^d)=\{n\in \Z: \tau_d^nx^{(d)}\in U\times \ldots\times U \}$$
is piecewise syndetic. We conjecture that this result is sharp. That is, we have the following

\begin{ques} Is there a minimal PI-system $(X,T)$ such that $(x,x)$ is not
$\F_{ps}$ recurrent for any $x\in X$ under $T\times T^2$?
\end{ques}

Question 1 and Question 2 are very closely related as the following theorem shows.

\begin{thm} If there is $x$ in a PI minimal system $(X,T)$ such that $(x,x)$ is $\F_{ps}$-recurrent under $T\times T^2$,
then $(x,x)$ is minimal under $T\times T^2$. That is, the positive answer to Question 1 implies the positive answer to Question 2.
\end{thm}
\begin{proof}
It follows from Theorem C, since the closure of a $\F_{ps}$-recurrent point is an $M$-subsystem.
\end{proof}

For a doubly minimal system $(X,T)$ and $d\in\N$, we know that $(X^d,\tau_d)$ is minimal. We would like to
know if a similar result holds for some PI-system. That is, we have

\begin{ques} Is there a minimal PI system $(X,T)$ which is a non-trivial open proximal extension of an
equicontinuous system $(Y,T)$ such that for each $x\in X$, $(x,x)$ is minimal for $T\times T^2$?
\end{ques}

We remark that by the main result in \cite{GHSWY}, it is easy to show that there is a dense $G_\delta$ set
$Y_0$ of $Y$ such that for each $y\in Y_0$ there is a dense $G_\delta$ set $X_y$ of $\pi^{-1}y$ with the property that
for each $x\in X_y$, $(x,x)$ is a transitive point of $W=:\overline{\O}((x,x),T\times T^2)$, where $\pi:X\lra Y$
is the factor map and $W=(\pi\times \pi)^{-1}(\pi\times \pi)W.$



\begin{thebibliography}{SS}

\bibitem{Ak97} E. Akin, \textit{ Recurrence in topological dynamical
systems: Furstenberg families and Ellis actions}, Plenum Press,
New York, 1997.

\bibitem{AM85} J. Auslander, N. Markley, \textit{Graphic flows and multiple disjointness},
Trans. Amer. Math. Soc. 292 (1985), no. 2, 483--499.


\bibitem{BS} Z. I. Borevich and I. R. Shafarevich, \textit{Number Theory}, Academic Press, 1966.


\bibitem{BSZ} J. Bourgain, P. Sarnak, and T. Ziegler, \textit{Disjointness of M\"{o}bius from horocycle flow}, From Fourier analysis and number theory to Radon transforms and geometry, Dev. Math., vol. 28, Springer, New York, 2013, pp. 67--83.


\bibitem{CTY} F. Garcia-Ramos, T. Tager and X. Ye: \textit{Mean equicontinuity, almost automorphy and regularity}, arXiv:1908.05207, to appear in Israel J. of Math.


\bibitem{EGS} R. Ellis, S. Glasner and L. Shapiro, \textit{Proximal-Isometric Flows}, Advances in Math {\bfseries 17}, (1975), 213-260.


\bibitem{F67} H. Furstenberg, \textit{Disjointness in ergodic theory, minimal sets,
and a problem in Diophantine approximation}, {Mathematical Systems Theory. An International
Journal on Mathematical Computing Theory}, {\bf 1} (1967), 1--49.

\bibitem{F73} H. Furstenberg, \textit{The unique ergodicity of the horocycle flow}, in Recent Advances in
Topological Dynamics, Springer-Verlag Lecture Notes in Math. 318, 1973,  95--115.

\bibitem{F77} H. Furstenberg, \textit{Ergodic behavior of diagonal measures and a theorem of Szemer\'edi on arithmetic progressions}. J. Analyse Math. 31 (1977), 204--256.

\bibitem{F} H. Furstenberg, \textit{Recurrence in ergodic theory and combinatorial number theory}, M. B. Porter Lectures. Princeton University Press, Princeton, N.J., 1981.

\bibitem{F81} H. Furstenberg, \textit{Poincar\'e recurrence and number theory}, \text{ Bull. Amer. Math. Soc.}  5 (1981), no. 3, 211--234.


\bibitem{FW} H. Furstenberg and B. Weiss, \textit{Topological dynamics and combinatorial number theory}, Journal d'Analyse Math. 34 (1978), 61--85.


\bibitem{G76} S. Glasner, \textit{Proximal flows}, Lecture Notes in Mathematics, Vol. 517. Springer-Verlag, Berlin-New York, 1976.


\bibitem{G94} E. Glasner, \textit{Topological ergodic decompositions and
applications to products of powers of a minimal transformation}, J.
Anal. Math. 64 (1994), 241--262.


\bibitem{GHSWY} E. Glasner, W. Huang, S. Shao, B. Weiss and X. Ye, \textit{Topological characteristic factors and nilsystems}, arXiv:2006.12385 [math.DS], preprint, 2020.


\bibitem{GW94} E. Glasner and B. Weiss, \textit{A simple weakly mixing transformation with nonunique prime factors}, Amer. J. Math. 116 (1994), no. 2, 361--375.

\bibitem{GH} W. Gottschalk, G. Hedlund, \textit{Topological dynamics}, American Mathematical Society Colloquium Publications, Vol. 36. American Mathematical Society, Providence, R. I., 1955. vii+151 pp.


\bibitem{HY05} W. Huang and X. Ye, \textit{Dynamical systems disjoint from any minimal system}, {Trans. Amer. Math. Soc.}, {\bf 357} (2) (2005) 669--694.

\bibitem{HSY-19-1} W. Huang, S. Shao and X. Ye, \textit{Topological correspondence of multiple ergodic averages of nilpotent actions}, J. Anal. Math.  {\bf 138}(2019),  no. 2, 687-715.


\bibitem{KLOY} D. Kwietniak, J. Li, P. Oprocha and X. Ye,  \textit{Multi-recurrence and van der Waerden systems}, Science in China, Mathematics, {\bf 60} (2017), no. 1, 59-82.


\bibitem{Morris2005} D. Witte Morris,  \textit{Ratner's Theorems on Unipotent Flows}, University of Chicago Press, 2005.

\bibitem{Morris} D. Witte Morris,  \textit{Introduction to Arithmetic Groups}, 	arXiv:math/0106063


\bibitem{PR} V. Platonov and A. Rapinchuk, \textit{Algebraic Groups and Number Theory}, Academic Press, Boston, 1994.

\bibitem{R83} M. Ratner, \textit{Horocycle flows, joinings and rigidity of products}, Ann. of Math. (2) 118 (1983), no. 2, 277--313.

\bibitem{R91} M. Ratner, \textit{On Raghunathan's measure conjecture}, Ann. of Math. (2) 134 (1991), no. 3, 545--607.

\bibitem{R912} M. Ratner: \textit{Raghunathan's topological conjecture and distributions of unipotent flows}, Duke Math. J. 63 (1991), no. 1, 235--280.


\bibitem{Weiss95} B. Weiss, \textit{Multiple recurrence and doubly minimal systems},
Topological dynamics and applications (Minneapolis, MN, 1995),
189--196, Contemp. Math., 215, Amer. Math. Soc., Providence, RI,
1998.


\bibitem{Wo85} J. van der Woude, \textit{Characterizations of (H)PI extensions}, Pacific J. of Math., vol.120,  (1985), 453-467.

\bibitem{Zimmer} R. J. Zimmer, \textit{Ergodic Theory and Semisimple Groups}, Birkh\"{a}user, Boston, 1984.

\end{thebibliography}
\end{document}